\documentclass[a4paper,12pt]{article}

\usepackage{amssymb}

\usepackage{amsthm}
\usepackage{amsmath}
\theoremstyle{plain}
\newtheorem{theorem}{Theorem}[section]
\newtheorem{lemma}[theorem]{Lemma}
\newtheorem{corollary}[theorem]{Corollary}

\title{Relating log-tangent integrals with the Riemann zeta function}
\author{Lahoucine Elaissaoui\footnote{Corresponding author, Mohammed V University in Rabat, Faculty of sciences, Morocco; \texttt{lahoumaths@gmail.com}} \and Zine El-Abidine Guennoun\footnote{Mohammed V University in Rabat, Faculty of sciences, Morocco; \texttt{guennoun@fsr.ac.ma}}}

\begin{document}
\maketitle
\begin{abstract}
We show that integrals involving log-tangent function, with respect to certain square-integrable functions on $(0, \pi/2)$, can be evaluated by some series involving the harmonic number. Then we use this result to establish many closed forms relating to the Riemann zeta function at odd positive integers. In addition, we show that the log-tangent integral with respect to the Hurwitz zeta function defines a meromorphic function and that its values depend on the Dirichlet series $\zeta_h(s) :=\sum_{n = 1}^\infty h_n n^{-s}$, where $h_n = \sum_{k=1}^n(2k-1)^{-1}$.
\end{abstract}



\section{Introduction}	

The theory of Riemann zeta function is a fascinating topic in number theory with many beautiful results and open questions. The Riemann zeta function is  defined by $$\zeta(s) := \sum_{n = 1}^\infty \frac{1}{n^s}$$ for any complex complex number $s$ in the half-plane $\Re s > 1$ and by analytic continuation on the whole complex plane, except for the simple pole at $s=1$. The values of $\zeta(s)$ at odd positive integers still present a mystery, except for $\zeta(3)$ which was established in \cite{Ape} to be an irrational number. Moreover, Rivoal~\cite{Riv} and Zudilin~\cite{Zud} showed, respectively, that infinitely many of the numbers $\zeta(2n+1)$ ($n=1,2,\dots$) must be irrational, and that at least one of the eight numbers $\zeta(2n+1)$ ($n=2,\dots, 9$) must be irrational. In the same direction and in order to study some algebraic properties of the set $\{\zeta(2n+1), n \in \mathbb{N}\}$, the current authors \cite{Ela} showed recently that these numbers appear, both explicitly and implicitly, in the value of integrals involving log-tangent function. More precisely, we showed that, for any square-integrable function $f$ on $(0, \pi/2)$, the integral
$$ L(f) :=\int_0^{\frac{\pi}{2}} f(x) \log( \tan x ) \mathrm{d} x$$
can be approximated by a finite sum involving the Riemann zeta function at odd positive integers. In particular, for any polynomial $P$, we have \cite[Theorem 2]{Ela}
\begin{equation}
 \int_0^{\frac{\pi}{2}}P\left( \frac{2}{\pi} x\right) \log ( \tan x)\mathrm{d}x = \sum_{k=1}^{\lfloor \frac{\deg P +1}{2} \rfloor} \frac{(-1)^{k-1}}{\pi^{2k-1}}c_k(P) \zeta(2k+1),
\label{L(P)}
\end{equation}
where $$c_k(P) = \left( 1 - \frac{1}{2^{2k+1}}\right)\left(P^{(2k-1)}(1) + P^{(2k-1)}(0)\right)$$
and where $P^{(k)}(x)$ denotes the $k$th derivative of $P$ at point $x.$

Besides the theory of numbers, integrals involving log-tangent function have various important applications in many different fields of mathematics. In physics, logarithmic-trigonometric integrals also have some applications in the evaluation of classical, semi-classical and quantum entropies of position and momentum, see for example \cite{Ruiz}.

The purpose of this paper is to show that the integrals involving log-tangent function can be evaluated by some series regarding the harmonic number. That is, for certain square-integrable functions $f$ on $(0 , \pi/2),$ we have
$$L(f) = -\sum_{n = 1}^\infty b_n(f) \frac{h_n}{n}, $$
where
$$b_n(f) := \langle f , \sin(4n\cdot)\rangle = \frac{4}{\pi} \int_0^{\frac{\pi}{2}}f(x) \sin(4nx) \mathrm{d} x \qquad (n \in \mathbb{N}) $$
and where
$$h_n := \sum_{k=1}^n \frac{1}{2k-1} = H_{2n} - \frac{1}{2} H_n \qquad (n \in \mathbb{N}).$$
Here, $H_n := \sum_{k=1}^n \frac{1}{k}$ is the so-called harmonic number. Furthermore, by combining this result with those obtained in \cite{Ela}, we will be able to express many different sums involving the Riemann zeta-function at odd integer in terms of variant Euler sums,
\begin{equation}
\sum_{k = 1}^\infty \alpha_k \frac{h_k}{k},
\label{Eulsum}
\end{equation}
where $(\alpha_k)_{k= 1}^\infty$ is a real sequence. These mentioned results will appear in Section \ref{Section1}. It is worth noting that Chen~\cite{Chen} studied the Euler series \eqref{Eulsum}, when $(\alpha_k)_{k=1}^\infty$ is a power sequences, by utilizing the following generating function
\begin{equation}
\sum_{k = 1}^\infty \frac{h_k}{k}x^{2k} = \frac{1}{4} \log^2 \left( \frac{1+x}{1-x} \right) \qquad (|x|<1) .
\label{GF}
\end{equation}
In Section 3, we provide a general study of the Dirichlet series $\zeta_h(s)$, or the $h$-zeta function, defined by
$$\zeta_h(s) := \sum_{n=1}^{\infty} \frac{h_n}{n^s} \qquad (\Re s > 1) .$$
Finally, we conclude the paper with a brief discussion of the corresponding results and remarks.

%%%%%%%%%%%%%%%%%%%%%%%%%%%%%%%%%%%%%%%%%%%%%%%%%%%

\section{Recursive formulas involving zeta numbers}
\label{Section1}

It is well known that the Fourier system $\left\{1/\sqrt{2},\cos(4nx) , \sin(4nx)\right\}_{n =1}^\infty$ forms an orthonormal basis for the Hilbert space $\mathrm{L}^2([0, \pi/2])$ with respect to the inner product
$$\langle f , g \rangle = \frac{4}{\pi} \int_0^{\frac{\pi}{2}} f(x) g(x) \mathrm{d}x \qquad \left(f,g \in \mathrm{L}^2 \left(\left[0 , \frac{\pi}{2} \right]\right)\right) .$$
It follows that any square-integrable function $f$ on $[0, \pi/2]$ can be expressed as the Fourier series
\begin{equation}
f(x) = \frac{1}{\sqrt{2}}a_0(f) + \sum_{n= 1}^\infty a_n(f) \cos(4nx) + \sum_{n= 1}^\infty b_n(f) \sin(4nx),
\label{Expf}
\end{equation}
where $a_0(f) := \langle f , 1/\sqrt{2}\rangle$ , $a_n(f) := \langle f , \cos(4n \cdot)\rangle$ and $ b_n(f) := \langle f , \sin(4n \cdot)\rangle$.

Let
$$H_{\text{odd}}:= \left\{ f \in L^2\left(\left[ 0 , \frac{\pi}{2} \right]\right)\colon  f\left( \frac{\pi}{2}-x \right) = - f(x) \ \text{almost everywhere}\right\} $$
and
$$ H_{\text{even}}:= \left\{ f \in L^2\left(\left[ 0 , \frac{\pi}{2} \right]\right)\colon f\left( \frac{\pi}{2}-x \right) =  f(x) \ \text{almost everywhere} \right\}.$$
It is not hard to check that $H_{\text{even}}$ is the orthogonal complement subspace of the closed subspace $H_{\text{odd}}$ in $\mathrm{L}^2( [0 , \pi/2] )$ and that the system $\{ \sin(4nx) \}_{n= 1}^\infty$ forms an orthonormal basis for $H_{\text{odd}}$.

Recall that, for any square-integrable function $f$ on $(0,\pi/2),$
$$L(f):= \int_{0}^{\frac{\pi}{2}} f(x)\log(\tan x)\mathrm{d}x.$$
The function $x \mapsto \log(\tan x)$ belongs to $H_{\text{odd}}$ and the operator $L$ is a linear functional on $\mathrm{L}^2([0, \pi/2]),$ so we have
$$L\left(\mathrm{L}^2\left(\left[0 , \frac{\pi}{2}\right] \right)\right) = L\left(H_{\text{odd}} \right) = \mathbb{R}.$$
Therefore, it suffices to study the linear functional $L$ on the Hilbert space $H_{\text{odd}}$. Clearly, if $f \in H_{\text{odd}}$, then the Fourier expansion of $f$ is reduced to
$$f(x) = \sum_{n = 1}^\infty b_n(f) \sin(4nx). $$
It follows immediately that
$$L(f) = \sum_{n = 1}^\infty b_n(f) L(\sin(4n \cdot))$$
for any $f\in H_{\text{odd}}$.

\begin{lemma}
For any positive integer $n,$ we have
$$\int_0^{\frac{\pi}{2}}\sin(4nx) \log ( \tan x ) \mathrm{d}x = - \frac{h_n}{n}.$$
\label{Lm1}\end{lemma}
\begin{proof}
 By using the Fourier expansion of log-tan function from \cite[Th. 1]{Bra}, namely
$$\log(\tan x ) = -2 \sum_{k = 0}^\infty \frac{\cos(2(2k+1)x)}{2k+1} \qquad \left(x \in \left(0, \frac{\pi}{2} \right)\right),$$
and the fact that
$$\int_0^{\frac{\pi}{2}}\cos(2(2k+1)x)\sin(4nx) \mathrm{d}x = \frac{1}{2} \left( \frac{1}{2k+1+2n} - \frac{1}{2k+1-2n} \right),$$
we have
\begin{align*}
\int_0^{\frac{\pi}{2}}\sin(4nx) \log ( \tan x ) \mathrm{d}x &=  - \sum_{k = 0}^\infty \frac{1}{2k+1}\left( \frac{1}{2k+1+2n} - \frac{1}{2k+1-2n}\right) \\ &= -\frac{1}{2n} \sum_{k = 0}^\infty\left( \frac{1}{2k+1} - \frac{1}{2k+1+2n}\right) \\& \qquad \qquad - \frac{1}{2n} \sum_{k = 0}^\infty \left( \frac{1}{2k+1} - \frac{1}{2k+1-2n}\right)\\ &= -\frac{1}{4n}\left( \psi\left( \frac{1}{2} + n \right) + \psi\left( \frac{1}{2}-n\right) - 2\psi\left(\frac{1}{2}\right) \right),
\end{align*}
where $\psi$ is the digamma-function defined by
$$\psi(x) := -\frac{1}{x}- \gamma + \sum_{k \geq 1}\left( \frac{1}{k} - \frac{1}{x+k} \right) \qquad (x >0) $$
with the Euler-Mascheroni constant $\gamma.$
Since, for each $n \in \mathbb{N}$, we have
$$\psi\left( n + \frac{1}{2}\right) = \psi\left(\frac{1}{2}\right) + 2 h_n \qquad\mbox{and}\qquad \psi\left( \frac{1}{2} - n\right) = \psi\left( \frac{1}{2} + n \right), $$
it follows that
$$\int_0^{\frac{\pi}{2}}\sin(4nx) \log ( \tan x ) \mathrm{d}x = - \frac{h_n}{n}, $$
as required.
\end{proof}

Hence, we are ready to state our main theorem.

\begin{theorem}
The equality
 $$\log( \tan x) =  - \frac{4}{\pi}\sum_{n =1}^\infty \frac{h_n}{n} \sin(4nx) $$
 holds in $\mathrm{L}^2([0 , \pi/2])$ and, for any $f \in H_{\text{odd}},$ we have
 $$L(f) = - \sum_{n= 1}^\infty b_n( f ) \frac{h_n}{n},$$
 where $b_n(f):=\langle f,\sin(4n\cdot)\rangle.$
 \label{Th1}
 \end{theorem}

Specializing Theorem~\ref{Th1} for $f(x) = \log(\tan x)$, we obtain
$$L(f) = \int_0^{\frac{\pi}{2}}\log^2( \tan x ) \mathrm{d}x = \frac{4}{\pi}\sum_{n = 1}^\infty\frac{h_n^2}{n^2}. $$
We also know that
$$ \int_0^{\frac{\pi}{2}}\log^2( \tan x ) \mathrm{d}x = \frac{\pi^3}{8}.$$
Therefore, we must have
$$\sum_{n = 1}^\infty \frac{h_n^2}{n^2} = \frac{\pi^4}{32} .$$
Note that this last formula can also be obtained by applying Parseval's identity.

Now let
$$f(x) = \begin{cases}\begin{array}{cl} - 1/2 &\quad\mbox{if}\quad 0 \leq x < \pi/4, \\  1/2 &\quad\mbox{if}\quad \pi/4 < x \leq \pi/2, \\ 0 &\quad\mbox{if}\quad  x=\pi/4.\end{array}\end{cases} $$
It is clear that $f$ belongs to $H_{\text{odd}}$, so we have
$$f(x) = \sum_{n=1}^\infty b_n(f) \sin(4nx) = \frac{1}{\pi} \sum_{n = 1}^\infty \frac{(-1)^n-1}{n} \sin(4nx). $$
By Theorem~\ref{Th1}, we obtain
$$L(f) = - \frac{1}{\pi} \sum_{n = 1}^\infty ((-1)^n - 1)\frac{h_n}{n^2} = \frac{2}{\pi}\sum_{n = 0}^\infty\frac{h_{2n+1}}{(2n+1)^2}= G,$$
where $G$ denotes Catalan's constant. From \cite[eq. (19) p. 5]{Chen}, we know that
\begin{equation}
\sum_{n = 1}^\infty \frac{h_n}{n^2} = \frac{7}{4} \zeta(3).
\label{Che}
\end{equation}
It now follows that
$$\sum_{n = 1}^\infty (-1)^n \frac{h_n}{n^2} = \frac{7}{4}\zeta(3) - \pi G \  \qquad \mbox{and} \qquad \sum_{n = 1}^\infty \frac{h_{2n}}{n^2} = \frac{7}{4} \zeta(3) - \frac{\pi}{2} G. $$

More generally, we have the following result.

\begin{corollary}
Let $r \in [0, 1/4]$. Then we have
$$\int_0^{r\pi} \log( \tan x) \mathrm{d}x =  -\frac{7}{4\pi}\zeta(3) + \frac{1}{\pi} \sum_{n = 1}^\infty \frac{h_n}{n^2}\cos( 4nr\pi). $$
\label{coro1}
\end{corollary}
\begin{proof}
 Let
$$f_r(x) = \begin{cases}\begin{array}{cl} - 1/2 &\quad\mbox{if}\quad 0 \leq x < r\pi, \\  0 & \quad\mbox{if}\quad r\pi \leq  x \leq ( 1/2 - r )\pi, \\ 1/2 &\quad\mbox{if}\quad  (1/2 - r )\pi < x \leq \pi/2.\end{array}\end{cases} $$
Clearly, $f_r$ belongs to $H_{\text{odd}}$. Since, for each $n \in \mathbb{N},$ we have
\begin{align*}
b_n(f) &= \frac{4}{\pi} \int_0^{\frac{\pi}{2}} f_r(x) \sin(4nx) \mathrm{d}x \\ &= - \frac{4}{\pi}\int_0^{r\pi}\sin(4nx) \mathrm{d}x \\ &=  \frac{1}{\pi}\left(\frac{\cos(4nr\pi) - 1}{n}\right),
\end{align*}
it follows from Theorem \ref{Th1} and \eqref{Che} that
\begin{align*}
L(f_r) &= - \frac{1}{\pi} \sum_{n = 1}^\infty\frac{h_n}{n^2}( \cos(4nr\pi) - 1 )  \\ &= \frac{1}{\pi}\sum_{n = 1}^\infty \frac{h_n}{n^2} - \frac{1}{\pi}\sum_{n = 1}^\infty \frac{h_n}{n^2}\cos(4nr\pi) \\ &= \frac{7}{4\pi}\zeta(3) - \frac{1}{\pi} \sum_{n = 1}^\infty\frac{h_n}{n^2}\cos(4nr\pi).
\end{align*}
Finally, the fact that
$$L(f_r) = -\int_0^{\pi r} \log(\tan x ) \mathrm{d}x \qquad \left( r \in \left[0, \frac{1}{4}\right]\right) $$
completes the proof of Corollary \ref{coro1}.
\end{proof}

We should mention that Bradley~\cite{Bra} studied the transformation
$$T(r) := \int_0^{r \pi} \log( \tan x)\mathrm{d}x, $$
for all $r \in [0, 1/2].$ He showed that $$T\left( \frac{1}{2} - r \right) = T(r) \qquad \left( r \in \left[0, \frac12\right]\right), $$
which can also be deduced from our Corollary \ref{coro1}. Moreover, by using the identities listed in \cite[p. 172]{Bra},  we can obtain several values of sums of the form
$$\sum_{n = 1}^\infty \frac{h_n}{n^2}\left(\sum_{r_i}k_i\cos(4nr_i\pi)\right) $$
for some rational numbers $r_i \in [0 , 1/4]$ and for some integers $k_i$. For example, by using the identity \cite[A. 23]{Bra}, we obtain
$$5T\left( \frac{3}{20} \right) = 5 T\left( \frac{1}{20}\right) + 2T\left( \frac{1}{4} \right), $$
and we can arrive at an alternative representation of the Ap\'{e}ry's constant
$$\zeta(3) = \frac{2}{7} \sum_{n = 1}^\infty \frac{h_n}{n^2}\left(5\cos\left( n \frac{\pi}{5}\right) - 5\cos \left(n \frac{3\pi}{5} \right) + 2(-1)^n\right) .$$

Another interesting Euler sum is given in the following corollary.

\begin{corollary}
Let $m$ be a positive integer. Then we have
$$\zeta_h(2m) =-\frac{1}{2} \sum_{k=1}^{m} \left(2^{2k+1} - 1\right) \zeta(2m-2k) \zeta(2k+1), $$
where $\zeta_h(s):= \sum_{n=1}^\infty h_n/n^s$ $(\Re s>1)$ and where $\zeta(0) = - 1/2$.
\label{coro2}
\end{corollary}
\begin{proof}
From \cite[eq. (1.3)]{LRV}, we know that, for any positive integer $m$, the ($2m-1$)th Bernoulli polynomial can be expanded as
$$B_{2m-1}(x) = \frac{2(-1)^m(2m-1)!}{(2\pi)^{2m-1}} \sum_{n = 1}^\infty \frac{1}{n^{2m-1}}\sin(2\pi nx) \qquad (0\leq x<1).$$
In other words, we have, for all $x \in [0, \pi/2 ),$
$$B_{2m-1}\left( \frac{2}{\pi}x\right)= \frac{2(-1)^m(2m-1)!}{(2\pi)^{2m-1}} \sum_{n = 1}^\infty \frac{1}{n^{2m-1}}\sin(4nx). $$
By applying Theorem \ref{Th1}, we obtain
\begin{equation}
\int_0^{\frac{\pi}{2}}B_{2m-1}\left( \frac{2}{\pi}x\right) \log( \tan x )  \mathrm{d}x = \frac{2(-1)^{m-1}(2m-1)!}{(2\pi)^{2m-1}} \sum_{n = 1}^\infty \frac{h_n}{n^{2m}}.
\label{Brn}
\end{equation}
On the other hand, we find from  \eqref{L(P)} that
$$\int_0^{\frac{\pi}{2}}B_{2m-1}\left( \frac{2}{\pi}x\right) \log( \tan x )  \mathrm{d}x = 2 \sum_{k=1}^m\frac{(-1)^{k-1}}{\pi^{2k-1}}B_{2m-1}^{(2k-1)}(0) \left(1 - \frac{1}{2^{2k+1}} \right)\zeta(2k+1).$$
Note that
$$B_{2m-1}^{(2k-1)}(0) = \frac{(2m-1)!}{(2m-2k)!}B_{2m-2k},$$
where $B_{2m-2k}$ is the ($2m-2k$)th Bernoulli number
$$B_{2m-2k} = - \frac{2(2m-2k)!(-1)^{m-k}}{(2\pi)^{2m-2k}}\zeta(2m-2k), $$
with $\zeta(0) = - 1/2$. Now we have
$$ \int_0^{\frac{\pi}{2}}B_{2m-1}\left( \frac{2}{\pi}x\right) \log( \tan x )  \mathrm{d}x = \frac{(-1)^m(2m-1)!}{(2\pi)^{2m-1}}\sum_{k=1}^m \left( 2^{2k+1}-1 \right) \zeta(2m-2k)\zeta(2k+1) .$$
Thus, by combining this last equation with \eqref{Brn}, we obtain
$$ \sum_{n = 1}^\infty \frac{h_n}{n^{2m}} = - \frac{1}{2} \sum_{k=1}^m \left( 2^{2k+1}-1 \right) \zeta(2m-2k)\zeta(2k+1),$$
as required.
\end{proof}

It is worth noting that Corollary \ref{coro2} can be rewritten as
$$\sum_{n = 1}^\infty\frac{h_n}{n^{2m}} = \left(\frac{2^{2m+1}-1}{4}\right) \zeta(2m+1) -\frac{1}{2} \sum_{k=1}^{m-1} \left(2^{2k+1} - 1\right) \zeta(2m-2k) \zeta(2k+1) \qquad (m\geq 2).$$
It is not hard to see that our Corollary \ref{coro2} is a generalization of Chen's formula \eqref{Che}.

For some historical background of this kind of recursive formulas, it is worth mentioning that Euler~\cite{Eul} proved the following formula
$$2 \sum_{n = 1}^\infty \frac{H_n}{n^m} = (m+2)\zeta(m+1) - \sum_{k=1}^{m-2}\zeta(k+1)\zeta(m-k) \qquad (m \geq 2),$$
In addition, Georghiou and Philippou~\cite{GerPhi} showed that
$$\sum_{n = 1}^\infty \frac{H_n}{n^{2m+1}} = \frac{1}{2}  \sum_{k=2}^{2m}(-1)^k\zeta(k)\zeta(2m+2-k) \qquad (m \in \mathbb{N}). $$
Another well-known recursive formula for zeta-function at even positive integers is given in \cite[p. 167]{ChoiSriv} as
$$(2n+1)\zeta(2n) = 2 \sum_{k=1}^{n-1}\zeta(2k)\zeta(2n-2k) \qquad (n\geq 2). $$
For more relevant explicit Euler sums, the reader is advised to see \cite{Bor1,Bor2}.

Next we show another explicit formulation of \eqref{L(P)}.

\begin{theorem}
For any polynomial $P$ such that $P(1-x) = - P(x),$ we have
$$\int_0^{\frac{\pi}{2}}P\left( \frac{2}{\pi} x \right) \log ( \tan x) \mathrm{d}x = 4 \sum_{k=1}^{\frac{\deg P +1}{2}} \frac{(-1)^{k}}{(2\pi)^{2k-1}}P^{(2k-2)}(0)\zeta_h(2k),$$
where $P^{(k)}(x)$ denotes the $k$th derivative of $P$ at point $x.$
\label{Th2}
\end{theorem}
\begin{proof}
 It is clear, from the functional equation of $P$, that $P$ is of odd degree, say $2m-1$ ($m \in \mathbb{N}$). Thus, by using integration by parts, it is easy to see that
$$\int_0^{\frac{\pi}{2}}P\left( \frac{2}{\pi} x  \right) \sin(4nx) \mathrm{d}x = \pi \sum_{k=1}^{m} \frac{(-1)^{k-1}}{(2n\pi)^{2k-1}}P^{(2k-2)}(0). $$
Since the polynomial $P\left(2/\pi \, x \right)$ is in $H_{\text{odd}},$ it follows that its expansion is
\begin{align*}
P\left( \frac{2}{\pi}x \right) &= 4 \sum_{n = 1}^\infty \left( \sum_{k=1}^{m} \frac{(-1)^{k-1}}{(2n\pi)^{2k-1}}P^{(2k-2)}(0)\right) \sin(4nx) \\ &= 4 \sum_{k=1}^m \frac{(-1)^{k-1}}{(2 \pi)^{2k-1}}P^{(2k-2)}(0) \sum_{n = 1}^\infty \frac{\sin(4nx)}{n^{2k-1}}.
\end{align*}
Finally, we apply Theorem \ref{Th1} to complete the proof.
\end{proof}

We note that another way to prove Theorem \ref{Th2} is by employing the following fact. Each polynomial $P$ of degree $2m-1$ such that $P(x)=-P(1-x)$ can be written as
 $$P(x) = \frac{2(-1)^m(2\pi)^{2m-1}}{(2m-1)!} \sum_{k=1}^m \frac{(-1)^{k-1}}{(2 \pi)^{2k-1}}P^{(2k-2)}(0) B_{2k-1}(x) \qquad (0\leq x \leq 1).$$
Then we apply \eqref{Brn} to prove Theorem \ref{Th2}.

Next if we take $P(x) = E_{2m-1}(x)$ to be the $(2m-1)$th Euler polynomial, we obtain the following inversion formula of Corollary \ref{coro2}.

\begin{corollary}
For any positive integer $m,$ we have
$$\zeta(2m+1) = \frac{8}{\pi^2} \frac{1}{2^{2m+1}-1} \sum_{k=1}^m \zeta_h(2k)\left(2^{2m-2k+2}-1 \right)\zeta(2m-2k+2). $$
\label{coro3}
\end{corollary}
\begin{proof}
By applying Theorem \ref{Th2} to the Euler polynomial $E_{2m-1}$, we obtain
$$\int_0^{\frac{\pi}{2}}E_{2m-1}\left( \frac{2}{\pi}x \right)\log( \tan x)\mathrm{d}x = 4 \sum_{k=1}^{m}\frac{(-1)^{k}}{(2\pi)^{2k-1}}\zeta_h(2k)E_{2m-1}^{(2k-2)}(0). $$
For each $1\le k \leq m$, we have
\begin{align*}
E_{2m-1}^{(2k-2)}(0) &= \frac{(2m-1)!}{(2m-2k+1)!} E_{2m-2k+1}(0)\\ &= \frac{4(2m-1)!(-1)^{m-k+1}}{\pi^{2m-2k+2}}\left(1 - \frac{1}{2^{2m-2k+2}} \right)\zeta(2m-2k+2).
\end{align*}
It follows that
\begin{equation}
L\left( E_{2m-1}\left( \frac{2}{\pi} \cdot \right)\right) = \frac{16(-1)^{m-1}(2m-1)!}{(2\pi)^{2m+1}} \sum_{k=1}^m \zeta_h(2k)\left(2^{2m-2k+2}-1 \right)\zeta(2m-2k+2).
\label{L(Em)}
\end{equation}
On the other hand, by \eqref{L(P)} (see \cite[Th. 1]{Ela}), we have
$$\int_0^{\frac{\pi}{2}}E_{2m-1}\left( \frac{2}{\pi}x \right)\log( \tan x)\mathrm{d}x = \frac{2(-1)^{m-1}(2m-1)!}{\pi^{2m-1}}\left( 1 - \frac{1}{2^{2m+1}}\right)\zeta(2m+1) .$$
Combining this equation with \eqref{L(Em)}, we complete the proof of Corollary~\ref{coro3}.
\end{proof}

It is well known that $\zeta(2n) = r_n \pi^{2n}$ ($n \in \mathbb{N}_0$), where $r_n = (-1)^{n+1}B_{2n}2^{2n-1}/(2n)!$ is a rational number. However, no such representation in terms of $\pi$ is known for the zeta function at odd arguments. It is conjectured that the number $\zeta(2n+1)/\pi^{2n+1}$ is transcendental for every integer $n \in \mathbb{N}$. In view of Corollary \ref{coro2} and Corollary \ref{coro3}, we can deduce that the numbers $\zeta(2n+1)$ and $\zeta_h(2m)$, where $n$ and $m$ are positive integers,  have the ``same" algebraic properties. This means that if $\mathcal{A}$ denotes the set of all algebraic numbers in $[0,1]$, we have $\zeta(2n+1)/\pi^{2n+1} \in \mathcal{A}$ for all $n \in \mathbb{N}$ if and only if $\zeta_h(2m)/\pi^{2m+1} \in \mathcal{A}$ for all $ m \in \mathbb{N}$. In addition, if we let
$$\alpha_n := \frac{\zeta(2n+1)}{\pi^{2n+1}}  \qquad\text{and} \qquad \beta_n:=\frac{\zeta_h(2n)}{\pi^{2n+1}} \qquad (n \geq 1),$$
then we have, by Corollary \ref{coro2},
$$\beta_n = -\frac{1}{2}\sum_{k=1}^n\left(2^{2k+1}-1\right) r_{n-k}\alpha_k \qquad (n \in \mathbb{N}) $$
and, by Corollary \ref{coro3},
$$\alpha_n = \frac{8}{2^{2n+1}-1}\sum_{k=1}^n \left( 2^{2n-2k+2}-1\right) r_{n-k+1} \beta_k \qquad (n \geq 1). $$
Furthermore, for any integer $n\geq 2,$ we have the following recursive formula
$$\sum_{k=1}^{n-1} \left(4\left(2^{2n-2k+1} -1 \right)r_{n-k+1}\beta_k - \left(2^{2k+1} - 1 \right)r_{n-k} \alpha_k \right) = 0. $$

Now, by \eqref{Brn} and the fact that
$$\frac{z e^{xz}}{e^z -1} = \sum_{n = 0}^\infty B_n(x) \frac{z^n}{n!} \qquad (|z|<2\pi),$$
we have
\begin{align*}
\int_0^{\frac{\pi}{2}} e^{2xz} \log( \tan x)\mathrm{d}x &= \frac{e^{\pi z}-1}{\pi z}  \sum_{n = 1}^\infty \frac{(\pi z)^n}{n!}\int_0^{\frac{\pi}{2}}B_n\left( \frac{2}{\pi}x\right) \log( \tan x)\mathrm{d}x \\ &= \frac{e^{\pi z}-1}{\pi z} \sum_{n = 1}^\infty \frac{(-1)^{n-1}}{4^{n-1}} \zeta_h(2n) z^{2n-1}
\end{align*}
for all $|z|<2$. We compare this last equation with \cite[p. 11]{Ela} to obtain
$$\sum_{n = 1}^\infty (-1)^{n-1}\left( 1 - \frac{1}{2^{2n+1}}\right)\zeta(2n+1) z^{2n-1} = \frac{1}{\pi} \tanh\left( \frac{\pi}{2}z\right) \sum_{n = 1}^\infty \frac{(-1)^{n-1}}{4^{n-1}} \zeta_h(2n) z^{2n-2}. $$
for all $|z|<1$, Therefore, if we put $z = i2x/\pi,$ we have
\begin{equation}
\tan x = \frac{S(x)}{C(x)}\qquad \left(|x| < \frac{\pi}{2}\right),
\label{tan}
\end{equation}
where
$$S(x) := \frac{1}{4}\sum_{n = 1}^\infty \left( 2^{2n+1}-1\right)\alpha_n x^{2n-1} \qquad \text{and} \qquad C(x) := \sum_{n = 1}^\infty \beta_n x^{2n-2}.$$
Note that we can also deduce Corollary \ref{coro2} and Corollary \ref{coro3} directly from \eqref{tan}.

Clearly, the integral $\int_0^{\pi/2}e^{2xz} \log ( \tan x)\mathrm{d}x$
exists for any complex number $z.$ Moreover, we have
\begin{equation}
 \int_0^{\frac{\pi}{2}}e^{2xz} \log ( \tan x) \mathrm{d}x = \frac{e^{\pi z} - 1}{\pi} \sum_{n= 1}^\infty \frac{h_n}{n^2 + ( \frac{z}{2})^2} \qquad (z \in \mathbb{C}).
 \label{expz}
\end{equation}
The case that $z= 2ik$ ($k \in \mathbb{Z}$) is included as $z$ tends to $2ik$; that is,
$$\int_0^{\frac{\pi}{2}}e^{4ikx} \log ( \tan x) \mathrm{d}x =  \lim_{z \to 2ik} \left(\frac{e^{\pi z} - 1}{\pi}\right) \sum_{n= 1}^\infty \frac{h_n}{n^2 + ( \frac{z}{2})^2} = - i \frac{h_{|k|}}{k}.$$
The proof of \eqref{expz} is based on the expansion of the function $x\mapsto \exp(2xz)$ as in \eqref{Expf}; namely,
$$\frac{\exp(2xz)}{e^{\pi z} - 1} = \frac{\sqrt{2}}{\pi z} + \frac{z}{2\pi}\sum_{n = 1}^\infty \frac{\cos(4nx)}{n^2 + ( \frac{z}{2})^2} - \frac{1}{\pi} \sum_{n = 1}^\infty \frac{n \sin(4nx)}{n^2 + ( \frac{z}{2})^2}.$$
Then Theorem \ref{Th1} helps complete the proof of \eqref{expz}. On the other hand, the authors showed in \cite[p. 11]{Ela} that
  \begin{align*}
  \int_0^{\frac{\pi}{2}} e^{2xz} \log ( \tan x) \mathrm{d}x &= \frac{e^{\pi z}+1}{4z} \left(\psi\left( \frac{1+iz}{2} \right)+\psi\left( \frac{1-iz}{2} \right) - 2 \psi\left( \frac{1}{2}\right) \right)\\ &= \frac{e^{\pi z}+1}{2z} \left(\psi\left( \frac{1+iz}{2} \right)- \psi\left( \frac{1}{2}\right) - i\frac{\pi}{2} \tanh\left( \frac{\pi}{2}z \right) \right),
  \end{align*}
where $\psi$ is the digamma function. Therefore, we obtain the following identity

$$\frac{\pi}{2z}\left( \psi\left( \frac{1+iz}{2} \right)- \psi\left( \frac{1}{2}\right) -  i\frac{\pi}{2} \tanh\left( \frac{\pi}{2}z \right) \right) = \tanh\left( \frac{\pi}{2} z \right) \sum_{n = 1}^\infty \frac{h_n}{n^2 + ( \frac{z}{2})^2}.$$

%%%%%%%%%%%%%%%%%%%%%%%%%%%%%%%%%%%%%%%%%%%%%%%%%%%

\section{ Analytic continuation of $\zeta_h(s)$  }

Let $s=\sigma + it$ ($\sigma, t \in \mathbb{R}$) be a complex number, we define the $h$-zeta function by
$$\zeta_h(s) := \sum_{n = 1}^\infty \frac{h_n}{n^s},$$
where $h_n:= \sum_{k=1}^n 1/(2k-1).$ Since $h_n = \log(4n)/2 + \gamma/2  + O( 1/n^2)$ as $n$ tends to $\infty,$ then $\zeta_h$ is analytic in the half-plane $\sigma >1$. On the other hand, it follows from the generating function \eqref{GF} that, for any real $x>0$,
$$\sum_{n = 1}^\infty \frac{h_n}{n} e^{-nx} = \frac{1}{4} \log^2\left( \tanh \frac{x}{4} \right). $$
Thus, for all $\sigma >1$, we have
\begin{align*}
\int_0^{\infty} \log^2( \tanh x ) x^{s-2} \mathrm{d}x &= 4\int_0^{\infty}\sum_{n = 1}^\infty \frac{h_n}{n} e^{-4nx} x^{s-2} \mathrm{d}x  \\ &= 4^{2-s}\int_0^{\infty} \sum_{n = 1}^\infty \frac{h_n}{n^s} e^{-x} x^{s-2} \mathrm{d}x \\ &= 4^{2-s} \zeta_h(s)\Gamma(s-1).
\end{align*}
That is, for all $\sigma > 1$, we have
$$\zeta_h(s) = \frac{4^{s-2}}{\Gamma(s-1)} \int_0^{ \infty} \log^2( \tanh x ) x^{s-2} \mathrm{d}x. $$
Since
$$\log^2( \tanh x) \sim 4 e^{-4x} \quad (\text{as} \ x \to \infty) $$
and
\begin{align*}
\log^2( \tanh x) &= \left( \log x + \sum_{n = 1}^\infty \left(\frac{1-2^{2n-1}}{n}\right) \left(\frac{2^{2n}B_{2n}}{(2n)!}\right)x^{2n} \right)^2 \\ &= \left(\log x + 2\sum_{n = 1}^\infty \frac{(-1)^n}{n}\left(2^{2n-1}-1\right)r_n x^{2n} \right)^2 \\ &= \log^2 x + 4 \log x \sum_{n = 1}^\infty  w_n  x^{2n} + 4\sum_{n= 2}^\infty c_n x^{2n}
\end{align*}
for all $|x|<\pi$, where
$$w_n = \frac{(-1)^n}{n}\left(2^{2n-1}-1\right)r_n \qquad (n \in \mathbb{N}) $$
and where
$$c_n = (-1)^n \sum_{k=1}^{n-1} \frac{\left(2^{2k-1}-1 \right)\left( 2^{2n-2k-1}-1\right)}{k(n-k)} r_kr_{n-k}, \qquad (n=2,3,\dots),  $$
it follows that the integral $$ \int_0^{\infty} \log^2( \tanh x ) x^{s-2} \mathrm{d}x $$
is absolutely convergent for all $\sigma > 1$. Moreover, we have
$$\int_0^1 \log^2( \tanh x ) x^{s-2} \mathrm{d}x =  \frac{2}{(s-1)^3} -4  \sum_{n = 1}^\infty \frac{w_n}{(s+2n-1)^2} + 4 \sum_{n= 2}^\infty \frac{c_n}{s+2n-1},$$
for all $\sigma > 1$. Therefore,
$$\zeta_h(s) = \frac{4^{s-2}}{\Gamma(s)}\frac{2}{(s-1)^2} - \frac{4^{s-1}}{\Gamma(s-1)}\sum_{n = 1}^\infty \frac{w_n}{(s+2n-1)^2} + \frac{4^{s-1}}{\Gamma(s-1)}\sum_{n= 2}^\infty \frac{c_n}{s+2n-1} + \frac{4^{s-2}}{\Gamma(s-1)}K(s), $$
where
$$K(s) := \int_1^{\infty} \log^2( \tanh x)x^{s-2} \mathrm{d}x .  $$
Since $K$ is a regular function, $\zeta_h(s)$ has an analytic continuation to the whole complex plane, except at the pole $s=1$ of order $2$ and at simple poles $s= 1-2n$ ($n \in \mathbb{N}$). Moreover, the residue of $\zeta_h(s)$ at $s=1$ equals $- \psi(1/2)/2= \log 2 + \gamma/{2}$, where $\gamma := \lim_{n\to \infty} (H_n - \log n)=0.57721\cdots$ denotes the Euler-Mascheroni constant. For each positive integer $n,$ the residue of $\zeta_h(s)$ at $s=1-2n$ is equal to $-w_n 4^{-2n}(2n)! =- B_{2n}( 1/2)/(4n)$. Similar to the Riemann zeta function, $\zeta_h(s)$ vanishes at the negative even integers.

It is interesting to mention that the $h$-zeta function is closely related to the Hurwitz zeta function. Their first connection is that, for all $\sigma > 1$,
$$\zeta_h(s) = \sum_{n = 1}^\infty \frac{\zeta(s,n)}{2n-1}. $$
Recall that the Hurwitz zeta function is defined for all $\sigma > 1$ and for all $x>0$ by
$$\zeta(s,x) = \sum_{n= 0}^\infty \frac{1}{(n+x)^s}, $$
and it can be extended by analytic continuation to the whole complex plane, except at the simple pole $s=1$ with residue $1$. Also, it is well known that \cite[p. 269]{Whit}, for all $\sigma<0,$
$$\zeta(s,x)= \frac{2\Gamma(1-s)}{(2\pi)^{1-s}}\left( \sin\left(\frac{\pi}{2}s \right)\sum_{n= 1}^\infty \frac{\cos(2\pi nx)}{n^{1-s}}+ \cos\left(\frac{\pi}{2}s \right)\sum_{n= 1}^\infty \frac{\sin(2\pi nx)}{n^{1-s}}\right). $$
Hence, by Theorem \ref{Th1}, we have, for all $\sigma < 0$,
$$\int_0^{\frac{\pi}{2}} \zeta\left(s, \frac{2}{\pi} x \right) \log ( \tan x ) \mathrm{d}x = - \frac{2 \Gamma(1-s)}{(2\pi)^{1-s}}\cos\left( \frac{\pi}{2}s\right) \zeta_h(2-s). $$
However, the right-hand side of the equality above exists for all $\sigma < 1$. Therefore, for each $\sigma > 1$, we have
\begin{equation}
\zeta_h(s) = \frac{(2\pi)^{s-1}}{2\Gamma(s-1)\cos( \frac{\pi}{2}s)} \int_0^{\frac{\pi}{2}}\zeta\left( 2-s, \frac{2}{\pi}x\right)\log ( \tan x) \mathrm{d}x.
\label{ACH}
\end{equation}
Note that the case when $s=2n+1$ ($n \in \mathbb{N}$) can be treated as
\begin{equation}
\zeta_h(2n+1) = \frac{(-1)^n 2^{2n} \pi^{2n-1}}{(2n-1)!} \int_0^{\frac{\pi}{2}}\zeta'\left(1-2n, \frac{2}{\pi}x\right) \log( \tan x ) \mathrm{d}x.
\label{zhodd}
\end{equation}
Note also that \eqref{zhodd} can also be shown by applying Theorem \ref{Th1} with \cite[Proposition 2]{Adam}. In fact, for any positive integer $n,$
$$\int_0^{\frac{\pi}{2}}\zeta'\left(-n, \frac{2}{\pi}x\right) \log(\tan x) \mathrm{d}x = - \int_0^{\frac{\pi}{2}}\zeta'\left(-n, 1 - \frac{2}{\pi}x\right) \log(\tan x) \mathrm{d}x,  $$
so we have
$$\int_0^{\frac{\pi}{2}}\zeta'\left(-n, \frac{2}{\pi}x\right) \log(\tan x) \mathrm{d}x = \frac{1}{2} \int_0^{\frac{\pi}{2}}\left(\zeta'\left(-n, \frac{2}{\pi}x\right) - \zeta'\left(-n, 1 - \frac{2}{\pi}x\right)\right) \log(\tan x) \mathrm{d}x.$$
It now follows from \cite[Proposition 2]{Adam} and Theorem \ref{Th1} that, for odd integers,
\begin{align*}
\int_0^{\frac{\pi}{2}}\zeta'\left(1-2n, \frac{2}{\pi}x\right) \log(\tan x) \mathrm{d}x &=  i\frac{(-1)^n(2n-1)!}{2(2\pi)^{2n-1}}\int_0^{\frac{\pi}{2}} \mathrm{Li}_{2n}\left( e^{4ix}\right)\log( \tan x ) \mathrm{d}x \\ &= \frac{(-1)^n(2n-1)!}{2(2\pi)^{2n-1}} \sum_{k = 1}^\infty\frac{h_k}{k^{2n+1}}.
\end{align*}

It is worth remarking that $\zeta_h(s)$ can be extended analytically by using \eqref{ACH}. Indeed, Let $z$ be a complex number such that $z \neq 1$. For each $0< x < 1,$ we write
$$ \zeta(z,x)=  x^{-z} + \zeta^*(z,x).$$
Then it follows from \eqref{ACH} that the integral
$$ \int_0^{\frac{\pi}{2}}\zeta\left( z, \frac{2}{\pi}x\right)\log( \tan x )\mathrm{d}x $$
exists for all $\Re z < 1$. From \cite[1.518, eq. 3, p. 53]{Jef.Zwi}, we have
$$\log( \tan x )= \log x + 2 \sum_{k=1}^{\infty} \frac{2^{2k-1} - 1}{k} \zeta(2k) \left( \frac{x}{\pi}\right)^{2k} \qquad x\in \left( 0, \frac{\pi}{2} \right),$$
so, by using integration by parts, we obtain, for all $\Re z < 1,$
$$ \int_0^{\frac{\pi}{2}} \left( \frac{2}{\pi} x\right)^{-z} \log( \tan x) \mathrm{d}x = -\frac{\pi}{2} \frac{1}{(1-z)^2} + \frac{\pi}{2} \frac{\log( \frac{\pi}{2})}{1-z} + \frac{\pi}{2} \sum_{k=1}^{\infty} \left( 1 - \frac{1}{2^{2k-1}}\right)\frac{\zeta(2k)}{k(2k+1 -z)}  .$$
However, the right-hand side of the equality above converges absolutely for any complex number $z$ not equal to odd positive integers (i.e. $z \neq 1,3,5, \dots$), and this defines an analytic continuation of the function
$$z \mapsto \int_0^{\frac{\pi}{2}} \left( \frac{2}{\pi} x\right)^{-z} \log( \tan x) \mathrm{d}x .$$
On the other hand, the function $\zeta^*(z,x)$ is defined on the half-plane $ \Re z > 1$ by
$$\zeta^*(z,x) := \sum_{n=1}^{\infty} \frac{1}{(n+x)^z} = \frac{1}{\Gamma(z)} \int_0^{\infty}\frac{e^{-y}y^{z-1}}{1-e^{-y}}e^{-xy}\,d y \qquad (0<x<1).$$
It follows that, for all $\Re z > 1,$
 $$\int_0^{\frac{\pi}{2}} \zeta^*\left( z, \frac{2}{\pi}x\right)\log( \tan x )\mathrm{d}x = \frac{1}{\Gamma(z)} \int_0^{\infty}\frac{e^{-y}y^{z-1}}{1-e^{-y}}\int_0^{\frac{\pi}{2}}e^{-\frac{2}{\pi}xy}\log( \tan x) \mathrm{d}x \,d y. $$
Then, by \eqref{expz}, we obtain
 \begin{equation}
\int_0^{\frac{\pi}{2}} \zeta^*\left( z, \frac{2}{\pi}x\right)\log( \tan x )\mathrm{d}x = - \frac{1}{\pi}\frac{1}{\Gamma(z)} \int_0^{\infty} w(y)e^{-y}y^{z-1} \,d y
\label{HZ*}
\end{equation}
for all $\Re z > 1$, where
 $$w(y) = \sum_{n = 1}^\infty \frac{h_n}{n^2 + ( \frac{y}{2\pi} )^2}. $$
Note that the right-hand side integral in \eqref{HZ*} is absolutely convergent for all $\Re z \geq 1$. Consequently, the function
$$G(z) := \int_0^{\frac{\pi}{2}} \zeta\left( z, \frac{2}{\pi}x\right)\log( \tan x )\mathrm{d}x $$
can be continued analytically to the half-plane $\Re z \geq 1$, except at the poles $z=2n-1$ ($n \in \mathbb{N}$), by
\begin{align*} G(z) = -\frac{\pi}{2} \frac{1}{(1-z)^2} & + \frac{\pi}{2} \frac{\log (\frac{\pi}{2})}{1-z} + \frac{\pi}{2} \sum_{k=1}^{\infty} \left( 1 - \frac{1}{2^{2k-1}}\right)\frac{\zeta(2k)}{k(2k+1 -z)} \\ &- \frac{1}{\pi}\frac{1}{\Gamma(z)}\int_0^{\infty} w(y)e^{-y}y^{z-1} \,d y .
\end{align*}
Finally, we can define the $h$-zeta function on the whole complex plane by
$$\zeta_h(s):= \begin{cases}\begin{array}{cl}  \sum_{n = 1}^\infty h_n/n^s & \quad\mbox{if}\quad \sigma  > 1,  \\ -2^{s-1}\pi^{s-2}\sin( \pi s/2 )\Gamma(2-s)G(2-s) &\quad\mbox{if}\quad \sigma \leq 1\quad\mbox{and}\quad s\neq 1-2n \quad (n \in \mathbb{N}_0).\end{array}\end{cases} $$

%%%%%%%%%%%%%%%%%%%%%%%%%%%%%%%%%%%%%%%%%%%%%%%%%%%

\section{Concluding Remarks}

The value distribution of the $h$-zeta function can be an interesting topic to pursue. Besides, it is important to study the algebraic aspects of the numbers $\zeta(2n+1)$ as we have shown in Corollary \ref{coro2} and Corollary \ref{coro3} that the $h$-zeta function appears in several integrals involving the Riemann zeta-function and the Hurwitz zeta-function.

For the sake of completeness, we present the following interesting example.
For each $\sigma> 1$, let
$$H(s) := -\int_0^{\infty}\log( \tanh x ) x^{s-2} \mathrm{d}x. $$
Using integration by parts, we find that
$$H(s)= \frac{4}{s-1} \int_0^{\infty} \frac{e^{-2x}x^{s-1}}{1-e^{-4x}} \mathrm{d}x, $$
which implies that, for all $\sigma > 1,$
$$H(s) =  4^{1-s} \Gamma(s-1) \zeta\left(s , \frac{1}{2} \right). $$
On the other hand, the analogue of the Parseval's formula for the Mellin transforms (see for example \cite[p. 484]{Ivic}) yields
$$\frac{1}{2\pi } \int_{ \Re s = \sigma}\left|H(s) \right|^2 |\mathrm{d}s|= \int_0^{\infty}\log^2( \tanh x )x^{2\sigma - 3} \mathrm{d}x$$
for all $\sigma > 1$, which is equivalent to
 $$\frac{1}{8\pi} \int_{\Re s = \sigma}\left|\Gamma(s-1)\zeta\left(s, \frac{1}{2} \right) \right|^2 |\mathrm{d}s| = \Gamma(2\sigma -2) \zeta_h(2\sigma -1)$$
 for all $\sigma > 1$. In particular, when $\sigma =\frac{3}{2},$ we have
 $$\zeta(3) = \frac{1}{7} \int_0^{\infty}\frac{9 - 4\sqrt{2}\cos(t\log 2)}{\cosh(\pi t)}\left|\zeta\left( \frac{3}{2} + it \right) \right|^2 \,d t . $$

We have seen that the $h$-zeta function vanishes at $s=-2n$ ($n=0,1,2,\dots$) and that $\zeta_h(2n)$ is closely related to the numbers $\zeta(2n+1)$ ($n \in \mathbb{N}$). Hence, it is natural to ask if $\zeta_h(s)$ satisfies certain functional equation. Maybe, by answering this question, one can find a closed form for the numbers $\zeta(2n+1)$.

%%%%%%%%%%%%%%%%%%%%%%%%%%%%%%%%%%%%%%%%%%%%%%%%%%%

\section*{Acknowledgements}
The authors are very grateful to the anonymous referee for her or his instruction, comment and suggestion. Also, the authors would like to express their gratitude to Dr. Poj Lertchoosakul for the endless English correction.

\end{document}